\documentclass[12pt]{article}
\usepackage {graphics} 
\usepackage {graphicx}  
\usepackage {pstricks,pst-node,  pst-coil} 
\onecolumn  
\usepackage{amsmath,amssymb,latexsym}
\usepackage{amsmath,amsthm}
\newtheorem{tw}{Theorem}  
\newtheorem{lem}[tw]{Lemma}  
  
\newtheorem{cnj}[tw]{Conjecture}  
\newtheorem{cor}[tw]{Corollary}  
\newtheorem{rem}[tw]{Remark}

\begin{document}
\hyphenation{every}
\title{On decomposing multigraphs into locally irregular submultigraphs}
\author{Igor Grzelec\thanks{AGH University of Science and Technology, al. A. Mickiewicza 30, $30-059$  Krak\'ow, Poland, grzelec@agh.edu.pl}, Mariusz Woźniak\thanks{AGH University of Science and Technology, al. A. Mickiewicza 30, $30-059$  Krak\'ow, Poland, mwozniak@agh.edu.pl}}
\maketitle
\begin{abstract}
A locally irregular multigraph is a multigraph whose adjacent vertices have distinct degrees. The \textit{locally irregular edge coloring} is an edge coloring of a multigraph $G$ such that every color induces a locally irregular submultigraph of $G$. We say that a multigraph $G$ is \textit{locally irregular colorable} if it admits a locally irregular edge coloring and we denote by ${\rm lir}(G)$ the \textit{locally irregular chromatic index} of $G$, which is the smallest number of colors required in a locally irregular edge coloring of a locally irregular colorable multigraph $G$. We conjecture that for every connected graph $G$, which is not isomorphic to $K_2$, multigraph $^2G$ obtained from $G$ by doubling each edge admits ${\rm lir}(^2G)\leq 2$. This concept is closely related
to the well known 1-2-3 Conjecture, Local Irregularity Conjecture, (2, 2) Conjecture
and other similar problems concerning edge colorings. We show this conjecture holds for graph classes like paths, cycles, wheels, complete graphs, complete $k$-partite graphs and bipartite graphs. We also prove the general bound for locally irregular chromatic index for all 2-multigraphs using our result for bipartite graphs. \\

\end{abstract}
\section{Introduction}
All graphs and multigraphs considered in this paper are finite. Let $G=(V, E)$ be a graph. We call a function $f: E \rightarrow \{1, 2, \dots, k\}$ \textit{edge coloring} of $G$. We begin with presenting some methods for distinguishing neighboring vertices in $G$. For every vertex $x$ we put $\sigma (x):=\sum \limits_{x \in e}f(e)$. Two vertices $x$ and $y$ are \textit{distinguished} if $\sigma (x)\neq \sigma (y)$. We can interpret edge coloring of $G$ as creating multigraph from $G$ in which we replace each edge by $f(e)$ parallel edges. Then $\sigma (x)$ is a degree of the vertex $x$ in the multigraph $G'$ created from the graph $G$. If adjacent vertices have different degrees we call a multigraph \textit{locally irregular}.

Note that if all adjacent vertices are distinguished in the edge coloring of $G$, then the function $\sigma (x)$ defines a proper vertex coloring of $G$. Thus, we introduce a parameter $\chi_{\Sigma}(G)$, which is the smallest $k$ such that in edge coloring of $G$ all adjacent vertices are distinguished. We call such coloring \textit{neighbor-sum-distinguishing}. This problem was first introduced by Karoński, Łuczak and Thomason in \cite{Karonski Luczak Thomason}, where they also proposed the following conjecture.  
\begin{cnj}[1-2-3 Conjecture]
For every graph $G$ containing no isolated edges, $\chi_{\Sigma}(G)\leq 3$.
\end{cnj}
This conjecture remains still open, but there are some important results about the 1-2-3 Conjecture and we refer the reader to the survey \cite{Seamone}. The best known general result about this conjecture is that every graph containing no isolated edges admits $\chi_{\Sigma}(G)\leq 5$ and was proved by Kalkowski, Karoński and Pfender in \cite{Kalkowski Karonski Pfender}. In the case of regular graphs Przybyło proved in \cite{Przybylo new} that every $d$-regular graph $G$, where $d\geq 2$, admits $\chi_{\Sigma}(G)\leq 4$ and if $d\geq 10^8$ then $G$ admits $\chi_{\Sigma}(G)\leq 3$.

A weaker version of this neighbor distinguishing edge coloring is  \textit{multiset neighbor distinguishing} edge coloring. For each vertex $x$ from $G$ we denote by $M(x):=[f(e): x \in e]$ the multiset of colors of edges incident to the vertex $x$.  In this coloring two adjacent vertices $x$ and $y$ are distinguished if $M(x)\neq M(y)$. We define a parameter $\chi_M(G)$ as the smallest $k$ for which there exists a multiset neighbor distinguishing edge coloring of $G$. We can easily see that every graph $G$ satisfies $\chi_M(G) \leq \chi_{\Sigma}(G)$ because if the sums are different then the multisets are also different. The best known result about the multiset neighbor distinguishing edge coloring is the following theorem which was proved by  Vu\v cković in \cite{Vuckovic}.
\begin{tw}
For every graph $G$ containing no isolated edges, $\chi_M(G)\leq 3$.
\end{tw}

Every locally irregular graph $G$ admits $\chi_{\Sigma}(G)=\chi_M(G)=1$. This observation motivated a different approach to the problem of local  \hyphenation{ir-re-gu-la-ri-ty}irregularity of graphs. We denote by ${\rm lir}(G)$ the smallest number $k$ such that there exists a decomposition of graph $G$ into $k$ locally irregular graphs. We can easily see that not every graph has such decomposition. We define the family $\mathfrak{T}$ recursively as follows:
\begin{itemize}
  \item the triangle $K_3$ belongs to $\mathfrak{T}$,
  \item if $G$ is a graph from $\mathfrak{T}$, then any graph $G'$ obtained from $G$ by identifying a vertex $v\in V(G)$ of degree 2, which belongs to a triangle in $G$, with an end vertex of a path of even length or with an end vertex of a path of odd length such that the other end vertex of that path is identified with a vertex of a new triangle.
\end{itemize}
The family $\mathfrak{T'}$ consists of the family $\mathfrak{T}$, all odd length paths and all odd length cycles. In \cite{Baudon Bensmail Przybylo Wozniak} Baudon, Bensmail, Przybyło and Woźniak proved that only the graphs from the family $\mathfrak{T'}$ do not have decomposition into locally irregular graphs.

If the graph $G$ satisfies ${\rm lir}(G)\leq k$, then $\chi_M(G)\leq k$. This is true because in a decomposition of the graph $G$ into $k$ locally irregular graphs, every two neighboring vertices in $G$ have different degrees in at least one locally irregular graph. Therefore every two neighboring vertices in $G$ have multisets differing in the multiplicity of at least one element. Inspired by this fact Baudon, Bensmail, Przybyło and Woźniak in \cite{Baudon Bensmail Przybylo Wozniak} proposed the conjecture that every connected graph $G \notin \mathfrak{T'}$ satisfies ${\rm lir}(G)\leq 3$. However in 2021 Sedlar and \v Skrekovski in \cite{Sedlar Skrekovski} proved that the bow-tie graph $B$ presented in Figure \ref{bow-tie graph} is not decomposable into three locally irregular graphs. They also proposed the following new conjecture and asked if there are any other graphs which are not decomposable into three locally irregular graphs.
\begin{figure}[h!]
\centering
\includegraphics[width=4.5cm]{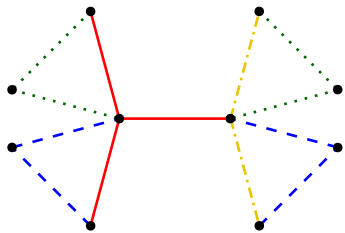}
\caption{The bow-tie graph $B$ and its decomposition into four locally irregular graphs.}
\label{bow-tie graph}
\end{figure}
\begin{cnj}[\cite{Sedlar Skrekovski}]
\label{graph4}
Every connected graph $G \notin \mathfrak{T'}$ satisfies ${\rm lir}(G)\leq 4$.
\end{cnj}
Perhaps the following version of the Local Irregularity Conjecture is true.

\begin{cnj}[\cite{Sedlar Skrekovski 2}]
\label{graph3}
Every connected graph $G \notin \mathfrak{T'}$ except for the bow-tie graph $B$  satisfies ${\rm lir}(G)\leq 3$.
\end{cnj}
Let us mention some results connected to Conjecture \ref{graph3}. This conjecture was proved for some graph classes among others trees \cite{Baudon Bensmail}, graphs with the minimum degree at least $10^{10}$ \cite{Przybylo}, $r$-regular graphs where $r\geq 10^7$ \cite{Baudon Bensmail Przybylo Wozniak} and cacti \cite{Sedlar Skrekovski 2}. For general connected graphs first Bensmail, Merker and Thomassen \cite{Bensmail Merker Thomassen} proved that 328 is the upper bound for ${\rm lir}(G)$ if $G\notin \mathfrak{T'}$. Later, the bound was lowered to the value of 220 by Lu\v zar, Przybyło and Soták  \cite{Luzar Przybylo Sotak}.

Another approach to the local irregularity of graph combine neighbour-sum-distinguishing edge coloring and graph decomposition into locally irregular graphs. Let $p$, $q$ be two positive integers. By $(p, q)$\textit{-coloring} of a graph $G$ we mean a decomposition of $G$ into at most $p$ subgraphs such that in each of these subgraphs the neighbouring vertices can be distinguished (by sums) using at most $q$ colors. We can easily see that the 1-2-3 Conjecture is equivalent to the statement that every graph containing no isolated edges admits (1, 3)-coloring. This notion was first introduced in \cite{Baudon Bensmail Davot Hocquard}, where Baudon et. al. proposed the following conjecture.

\begin{cnj}[(2, 2) Conjecture]
\label{2,2}
Every connected graph of order $n\geq 4$ has a $(2,2)$-coloring.
\end{cnj}

The above mentioned conjecture can be formulated in the language of multigraphs, but first we introduce some notation and terminology. Let $G$ be a graph. We denote by $\mathcal{M}(G)$ the family of all multigraphs created from $G$ by \textit{edge multiplication} i.e. an operation of replacing an edge $e=xy$ which is a set $\{x, y\}$ by a finite multiset $[\{x, y\}, \dots, \{x, y\}]$. Note that we do not need multiply all edges in $G$. We will denote by $\hat G$ a multigraph from the family $\mathcal{M}(G)$. Therefore we can treat a multigraph $\hat G \in \mathcal{M}(G)$ as a graph $G$ with additional function $\mu: E \rightarrow \{1, 2, \dots\}$ where $\mu(e)$ is the edge multiplicity. We shall also use the notation $\mu(e)=0$ to express the fact that $e \notin \hat G$.  We will denote by $\mathcal{M}^{[k]}(G)$ the family of all multigraphs created from $G$ by edge multiplication if multigraphs have edges with multiplicity at most $k$. By 2-$multigraph$ we mean a multigraph in which all edges have multiplicity equal to two and we denote it by $^2G$. Multigraph $\hat H$ is a \textit{submultigraph} of $\hat G$ if $H$ is a subgraph of $G$ and for each edge $e$ of $H$ holds $\mu_{\hat H}(e) \leq \mu_{\hat G}(e)$.  Analogically, multigraph $\hat H$ is an \textit{induced submultigraph} of $\hat G$ if $H$ is an induced subgraph of $G$ and for each edge $e$ of $H$ holds $\mu_{\hat H}(e)= \mu_{\hat G}(e)$. We denote by $\hat d(v)$ degree of the vertex $v$ in a multigraph (the number of single edges incident to the vertex $v$). We say that multigraphs $\hat G_1$ and $\hat G_2$ create the \textit{decomposition} of a multigraph $\hat G$ if for each edge $e$ from $G$ holds $\mu_{\hat G_1}(e)+ \mu_{\hat G_2}(e)= \mu_{\hat G}(e)$.
\begin{rem}
When we consider decomposition of a multigraph we often use the language of edge  coloring. When we decompose a multigraph into two multigraphs we use red-blue coloring i.e. we color the edges of the first multigraph red and the second blue. We denote by $\hat d_r(v)$ and by $\hat d_b(v)$ degree of the vertex $v$ in red and blue multigraph, representatively.
\end{rem}
Now we are ready to formulate the 1-2-3 Conjecture and the (2, 2) Conjecture in the language of multigraphs.

\begin{cnj}[1-2-3 Conjecture]
For every graph $G$ containing no isolated edges there exists a locally irregular multigraph $\hat G\in \mathcal{M}^{[3]}(G)$.
\end{cnj}

\begin{cnj}[(2, 2) Conjecture]
\label{2,2n}
Every connected graph $G$ of order \linebreak $n\geq4$ can be decomposed into two subgraphs $G_r$ and $G_b$ such that there exist locally irregular multigraphs $\hat G_r\in \mathcal{M}^{[2]}(G_r)$ and $\hat G_b\in \mathcal{M}^{[2]}(G_b)$. 
\end{cnj}

Before we present our conjecture we give a few definitions. The \textit{locally irregular edge coloring} is an edge coloring of a multigraph $M$ such that every color induces a locally irregular submultigraph of $M$. We say that a multigraph is \textit{locally irregular colorable} if it satisfies the locally irregular edge coloring. The \textit{locally irregular chromatic index} of a locally irregular colorable multigraph $M$, denoted by ${\rm lir}(M)$, is the smallest number of colors required in a locally irregular edge coloring of $M$. In this paper we focus on locally irregular edge coloring of 2-multigraph $^2G$ obtained from graph $G$ by doubling each edge.   
\begin{cnj}
\label{main}
For every connected graph $G$ which is not isomorphic to $K_2$ we have ${\rm lir}(^2G)\leq 2$.
\end{cnj}

\textbf{Remark.} Conjecture \ref{main} is independent from the (2, 2) Conjecture. In our Conjecture \ref{main} we allow multiedges which are colored both red and blue whereas in the (2, 2) Conjecture all elements of the multiedge have the same color. We say that multiedge is colored \textit{red-blue} if one element of the multiedge is red and the second is blue. Another difference is that in the \linebreak (2, 2) Conjecture we do not have to double all multiedges in the multigraph $G_r' \cup G_b'$. In particular, note that a decomposition of a cycle $C_3$ described by the (2, 2) Conjecture does not exist, but multigraph $^2C_3$ obtained from $C_3$ can be decomposed into two multigraphs because the following coloring of $^2C_3$: first multiedge red, second red-blue and third blue, is locally irregular (see Figure \ref{fCycles1}).

In this paper we will show in Section 2 that Conjecture \ref{main} is true for simple graph classes like paths, cycles, wheels, complete graphs and complete $k$-partite graphs. In Section 3 we will prove Conjecture \ref{main} for all bipartite graphs. Finally in Section 4 we will prove the general bound for locally irregular chromatic index for all connected 2-multigraphs which are not isomorphic to $^2K_2$ using similar method as in \cite{Bensmail Merker Thomassen} and our result for bipartite graphs. 

\section{Simple graph classes}
In this section we consider our conjecture for paths, cycles, wheels, complete graphs and complete $k$-partite graphs. We will denote by $P_n$ a path with $n$ vertices and by $W_n$ a wheel of order $n$, which consists of cycle of length $n-1$ and one central vertex connected with all vertices on the cycle. We will call a \textit{multicycle} a multigraph which is obtained from a cycle by doubling each edge. 

\begin{tw}
\label{cycle}
Conjecture $\ref{main}$ holds for paths, cycles and wheels.
\end{tw}
\begin{proof}
First, we consider multipaths $^2P_n$ of even length. We color first two multiedges blue, next two multiedges red and we repeat this color sequence to the end of the multipath.  
Then we consider multipaths of odd length, which are not isomorphic to $^2K_2$. We color first multiedge blue, second red-blue, third red and then we color remaining multiedges in the same way as multipath of even length.

\begin{figure}[h!]
\centering
\includegraphics[width=1\textwidth]{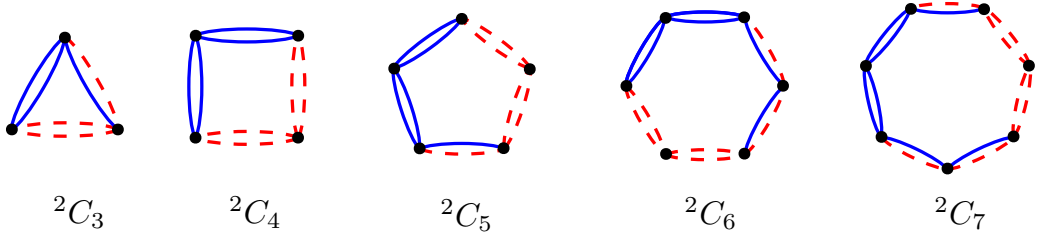}
\caption{Multicycles coloring.}
\label{fCycles1}
\end{figure}

First, we consider multicycles of length from three to seven. We color them as in Figure \ref{fCycles1}. The coloring of longer multicycle we obtain by adding multipath of length divisible by four colored in the same way as above to the appropriate colored multicycle of length from four to seven after two red multiedges.

We consider multigraph $^2W_n$ obtained from wheel $W_n$. First, we color the multicycle of length $n-1$ using the above method. Then, we color all incident multiedges to the central vertex red. Note that a central vertex in $^2W_n$ has greater degree than other vertices in $^2W_n$ for $n>4$. If $n=4$ we can easily see that all vertices have different red and blue degrees.
\end{proof}
\begin{tw}
Conjecture $\ref{main}$ holds for complete graphs, complete k-partite graphs, where $k\geq 2$.
\end{tw}
\begin{proof}
Assume that all multigraphs considered in this proof are not isomorphic to $^2K_2$. 

\textbf{Complete 2-multigraph.} We construct the coloring of this multigraph starting from the coloring of $^2C_3$ presented in Figure  \ref{fCycles1}. Then, we color blue all multiedges from the fourth vertex to vertices that have colored some incident multiedges. Next, we color red all multiedges from the fifth vertex to vertices that have colored some incident multiedges. Then, we color blue all multiedges from the sixth vertex to vertices that have colored some incident multiedges. We continue this procedure until we color the whole 2-multigraph.

\textbf{Complete $k$-partite 2-multigraph.} First, we assume that $k=2$ and we denote independent sets by $X$ and $Y$. We set $|X|=p$ and $|Y|=q$. If $p\neq q$ then 2-multigraph $^2K_{p, q}$ is locally irregular. On the opposite, if $p=q$ then we choose one vertex $v$ and we color all incident multiedges with $v$ red and we color all remaining multiedges blue. We can easily see that this coloring is locally irregular.

We assume that $k=3$ and we denote independent sets by $X$, $Y$, $Z$. We set $|X|=p$, $|Y|=q$, $|Z|=r$. If $p$, $q$, $r$ are pairwise distinct then 2-multigraph $^2K_{p, q, r}$ is locally irregular and we color all multiedges red. If $p=q \neq r$ then we color all multiedges from the set $X$ to $Z$ blue and we color all remaining multiedges in $^2K_{p, q, r}$ red. Thus, all vertices in this 2-multigraph: in $X$ have red degree equal to $2p$ and blue degree equal to $2r$, in $Y$ have red degree equal to $2p+2r$ and blue degree equal to 0, in $Z$ have red and blue degree equal to $2p$ therefore this coloring is locally irregular. We use analogical coloring when $p\neq q = r$ and $p=r \neq q$. If $p=q=r$ then we color all multiedges: from the set $X$ to $Z$ red, from the set $Y$ to $Z$ blue and from the set $X$ to $Y$ red-blue. Thus, all vertices: in $X$ have red degree equal to $3p$ and blue degree equal to $p$, in $Y$ have red degree equal to $p$ and blue degree equal to $3p$, in $Z$ have red and blue degree equal to $2p$ therefor this coloring is locally irregular.

We assume that $k>3$. We denote independent sets according to the increasing number of vertices by $A_1, \dots, A_k$. If two independent sets have the same number of vertices then we order them arbitrarily. First, we color induced submultigraph by sets $A_1, A_2, A_3$ using the same method as for complete $3$-partite 2-multigraphs from previous case. Then, we color all multiedges from the set $A_4$ to sets $A_1, A_2, A_3$ blue. Next, we color all multiedges from $A_5$ to sets $A_1, \dots, A_4$ red. Next, we color all multiedges from $A_6$ to sets $A_1, \dots, A_5$ blue. We continue this procedure until we color the whole 2-multigraph. We can easily see that this coloring is locally irregular.
\end{proof}

\section{Bipartite graphs}
First, we introduce notion and lemma which will be useful to prove our main result for bipartite graphs. Let $G$ be a graph. For a set $S$ of vertices, we put $N(S):= \bigcup\limits_{s\in S}N(s)$. By \textit{twins} we mean two vertices $x$ and $y$ such that $N(x) = N(y)$. Note that the relation of being a twin is reflexive. The following lemma was established in \cite{Havet Paramaguru Sampathkumar}. 

\begin{lem}
\label{twins}
Let $G=(X, Y; E)$ be a connected bipartite graph. Then there exists a nonempty set of twins $S$ such that $G-(S \cup N(S))$ is connected.
\end{lem}

Now we are ready to prove our main result for bipartite graphs.

\begin{tw}
\label{bipartite graph}
For every connected bipartite graph $G$ which is not isomorphic to $K_2$, the multigraph $^2G$ satisfies ${\rm lir}(^2G)\leq 2$.
\end{tw}
\begin{proof}
Let $G=(X, Y; E)$ be a connected bipartite graph. First, we consider the situation when $|X|$ or $|Y|$ is even. Assume that $|X|$ is even. Put $X=\{x_1, x_2, \dots, x_{2p}\}$. For every $i$, $1 \leq i \leq p$, let $P_i$ be a path joining $x_{2i-1}$ to $x_{2i}$ in $G$. We consider 
multigraph $^2G$. We start with all multiedges colored blue. By odd vertex we will call vertex which has odd red and blue degree, analogically by even vertex we will call vertex which has even red and blue degree. Then, for each $i$,  $1 \leq i \leq p$, we exchange colors along $P_i$. Thus, at the end of this process, every vertex in $X$ is odd and every vertex in $Y$ is even.  Thus, we get the claim in this case. We call this set of paths \textit{path-system with ends in} $X$. 

Assume that $|X|$ and $|Y|$ are odd. By Lemma \ref{twins}, there is a set $S$ of twins such that $G-(S \cup N(S))$ is connected. Without loss of generality we may assume that $S \subset X$. Note that the subgraph induced by $S \cup N(S)$ is complete bipartite. If we have more than one such set $S$ we take $S$ with the smallest $|S|+|N(S)|$. Thus, each vertex in $N(S)$ has a neighbour in $X \setminus S$. If this is not true we can take smaller set $S'$ of twins such that $G-(S' \cup N(S'))$ is connected, which is the subset of $N(S)$ and vertices from $N(S)$ which has neighbour not in $S$ are not in $S'$. Therefore, we get contradiction with the fact that $S$ has the smallest $|S|+|N(S)|$. Put $X':=X \setminus S$, $T:=N(S)$, $Y':=Y \setminus N(S)$, $s:=|S|$ and $t:=|T|$. Note that we do not have any edge between $S$ and $Y'$ in graph $G$ (see Figure \ref{f1}). We double all edges in graph $G$. We will consider two main caseses. \\
\textbf{Case 1:}\quad $s$ is odd. First, we consider the subcase when $s\ne t$. Notice that $|X'|$ is even. Thus, we color submultigraph induced by $X' \cup Y'$ in $^2G$ using the path-system with ends in $X'$. More precisely we color this path-system with ends in $X'$ red and the rest multiedges in this submultigraph blue. Then, we color all multiedges between the vertex set $S$ and $T$ blue and we color all multiedges between $T$ and $X'$ red (see Figure \ref{f1}).
\begin{figure}[h!]
\centering
\includegraphics[width=6.3cm]{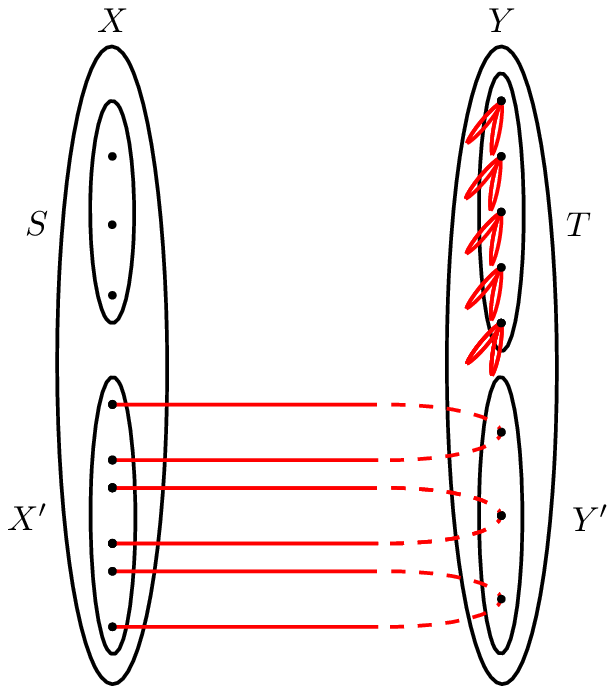}
\caption{The coloring of bipartite 2-multigraph $^2G$ in case 1, when $s\ne t$.}
\label{f1}
\end{figure}
We can easily see that this coloring of $^2G$ is locally irregular. Indeed blue multigraph have two components: multigraph induced by $S \cup T$ and multigraph induced by $X \cup Y'$ without path-system with ends in $X'$. From our assumption that $s\ne t$, blue multigraph induced by $S \cup T$ is locally irregular. Note also that in blue multigraph induced by $X \cup Y'$ without path-system with ends in $X'$ and red multigraph, all vertices in $X'$ are odd and all vertices in $Y$ are even.

We consider the situation when $s=t$. Notice that $|X'|$ is even. Thus, we color submultigraph induced by $X' \cup Y'$ in $^2G$ using the same method as in the situation when $s\ne t$. Then, we color all remaining multiedges in $^2G$ blue. Note that multigraph induced by $S \cup T$ is blue.
So, this coloring of $^2G$ is locally irregular, because all vertices in $S$ have blue degrees equal to $2t$ and are distinct from blue degrees of vertices in $T$ and all vertices in $X'$ are odd and in $Y$ even. Thus we are done. \\

\textbf{Case 2:}\quad $s$ is even. We will consider two main subcases. We denote by $x_0$ arbitrary vertex in $S$ and we take the vertex $y_0$ in $T$ in such a way that $z_0$ is a neighbour of $y_0$ in $X'$. We color multiedges $x_0y_0$ and $y_0z_0$ red-blue. Notice that $|X' \setminus \{z_0\}|$ is even. Thus, we color submultigraph induced by $(X' \setminus \{z_0\}) \cup Y'$ in $^2G$ using the path-system with ends in $X'\setminus \{z_0\}$. More precisely we color this path-system with ends in $X'\setminus \{z_0\}$ red and the rest of multiedges in this submultigraph blue. Then, we color all multiedges from the vertex $z_0$ to its neighbours in $Y'$ blue.  Note that path-system with ends in $X' \setminus \{z_0\}$ and path $x_0y_0z_0$ create path-system with ends in $X' \cup \{x_0\}$. This part of the coloring of $^2G$ is the same for all subcases.

\textbf{Subcase 2a:}\quad $s\neq t$. We color all multiedges between $T \setminus \{y_0\}$ and $X'$ red. Next we color all multiedges edges from $y_0$ to $X'$ except for $y_0z_0$ red and all remaining multiedges in $^2G$ blue. This coloring of $^2G$ is presented in Figure \ref{f2a}.  
\begin{figure}[h!]
\centering
\includegraphics[width=6.3cm]{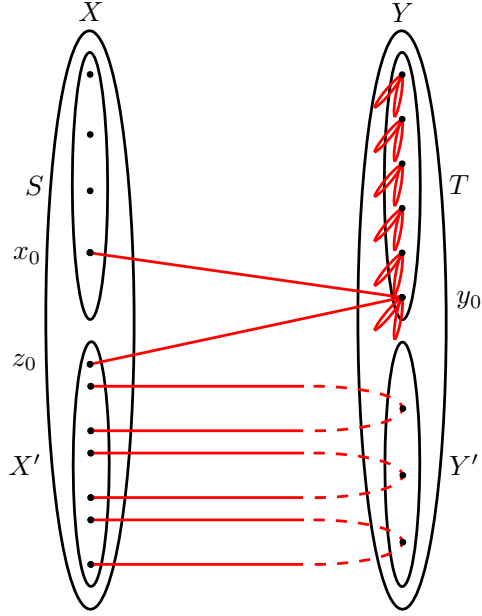}
\caption{The coloring of bipartite 2-multigraph $^2G$ in case 2, when $s\neq t$.}
\label{f2a}
\end{figure}

Notice that in this coloring of $^2G$ all vertices in $X' \cup \{x_0\}$ are odd and in $Y$ even. Note also that all vertices in $S\setminus \{x_0\}$ have blue degrees equal to $2t$ and all vertices in $T$ including $y_0$ have blue degrees equal to $2s$. Thus, this coloring of $^2G$ is locally irregular. 

\textbf{Subcase 2b:}\quad $s=t$. We start our coloring of $^2G$ from the common part for all subcases. Then, we color all multiedges between vertices from the set $T \setminus \{y_0\}$ and $X'$ blue. Next, we color all multiedges from $y_0$ to $X'$ except for $y_0z_0$ blue. At the end, we color all remaining multiedges in the 2-multigraph $^2G$ blue. This initial coloring of bipartite 2-multigraph $^2G$, when $s$ is even and $s=t$ is shown in Figure \ref{f2b}.  
\begin{figure}[h!]
\centering
\includegraphics[width=6.3cm]{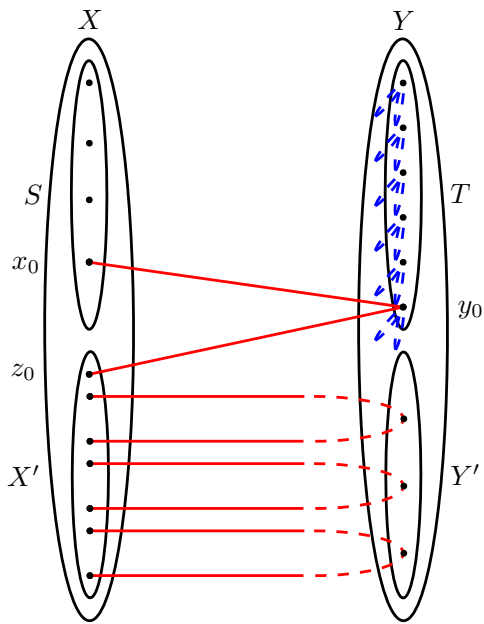}
\caption{The initial coloring of bipartite 2-multigraph $^2G$ in case 2, when $s=t$.}
\label{f2b}
\end{figure}

Note that in this coloring of $^2G$ all vertices in $X' \cup \{x_0\}$ are odd and in $Y$ even. Notice that each vertex $x$ from the set $S\setminus \{x_0\}$ has $\hat d_b(x)=2t$ and each vertex $y$ from the set $T\setminus \{y_0\}$ has $\hat d_b(y)\geq 2s+2$. We also see that $\hat d_b(y_0)\geq 2s$. If we have more than one multiedge between $y_0$ and the set $X'$, the vertex $y_0$ has $\hat d_b(y_0)\geq 2s+2$. Thus, in this situation we do not have conflict between vertices from $S\setminus \{x_0\}$ and $T$, therefore this coloring is locally irregular.

Now we consider the particular situation when $s=t$ and it is exactly one multiedge between $y_0$ and the set $X'$ for each $y_0\in T$. Let $y_t$ be an arbitrary vertex in $T$ distinct from $y_0$. We recolor all multiedges from the vertex $y_t$ to the set $S$ red in the initial coloring of bipartite 2-multigraph $^2G$, when $s$ is even and $s=t$ (see Figure \ref{f2c}). Thus, each vertex $x$ from the set $S\setminus \{x_0\}$ has $\hat d_b(x)=2t-2$ and each vertex $y$ from the set $T\setminus \{y_0, y_t\}$ has $\hat d_b(y)\geq 2s+2$. We also see that $\hat d_b(y_0)= 2s$. Note that we do not have conflicts caused by red degrees in $^2G$. Thus, we get our claim in this subcase.
\begin{figure}[h!]
\centering
\includegraphics[width=6.3cm]{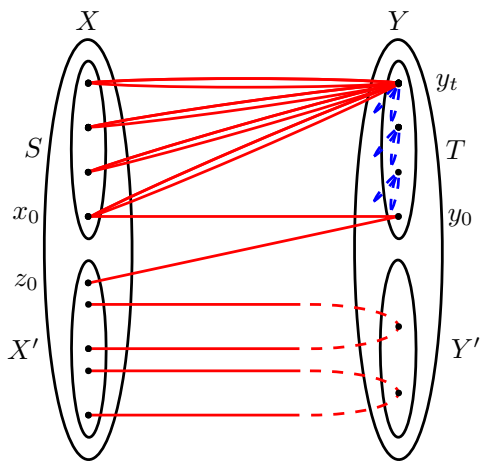}
\caption{The coloring of bipartite 2-multigraph $^2G$ in case 2, when $s=t$ and is exactly one multiedge between $y_0$ and the set $X'$.}
\label{f2c}
\end{figure}
\end{proof}

As an immediate consequence of the above theorem we get the following result. 
\begin{cor}
For every tree $T$ which is not isomorphic to $K_2$ we have ${\rm lir}(^2T)\leq 2$.
\end{cor}

\section{General bound for locally irregular \\ chromatic index for 2-multigraphs}
First, we prove the following lemma concerning the family $\mathfrak{T}$.

\begin{lem}
\label{family f}
For every graph $G$ from the family $\mathfrak{T}$, the multigraph $^2G$ satisfies ${\rm lir}(^2G)\leq 3$.
\end{lem}
\begin{proof}
It is easy to see that even length multipaths as well as odd length multipaths ended with a triangle $^2C_3$ can be decomposed into multipaths of length two. So, any multigraph $^2G$ with $G$ belonging to $\mathfrak{T}$ can be colored using three colors recursively as follows. 

The starting triangle we color with two colors as in Theorem \ref{cycle}. Next, for each multipaths we add to a triangle, we use two colors by starting by the color which does not appear on this triangle.
\end{proof}

\textbf{Remark.} One can prove that for every graph $G$ from the family $\mathfrak{T}$ the multigraph $^2G$ admits ${\rm lir}(^2G)\leq 2$ but this proof is technical and the above lemma is completely sufficient for us here.

Let us observe that if a graph $G$ is decomposable into $k$ locally irregular graphs then the multigraph $^2G$ is also decomposable into $k$ locally irregular multigraphs. Therefore, from Theorem \ref{cycle}, the above lemma and Bensmail, Merker and Thomassen result from \cite{Bensmail Merker Thomassen} we immediately have the existence of a constant upper bound equal to 328. 

However, repeating exactly the method from \cite{Bensmail Merker Thomassen} and using the fact that for bipartite graphs we have an upper bound equal to two (see Theorem \ref{bipartite graph}), and the authors of above mentioned paper had an upper bound equal to ten, we get the following result.

\begin{tw}
For every connected graph $G$ which is not isomorphic to $K_2$ we have ${\rm lir}(^2G)\leq 76$.  $\quad\quad\quad\quad\quad\quad\quad\quad\quad\quad\quad\quad\quad\quad\quad\quad\quad\quad\quad\quad\quad\quad \square$
\end{tw}

\end{document}